\documentclass[11pt]{article}
\usepackage[T1]{fontenc}
\usepackage{amsmath,amssymb,amsthm,mathtools}
\usepackage{enumitem}
\usepackage{microtype}
\usepackage[hidelinks]{hyperref}
\usepackage[margin=1.15in]{geometry}

\newtheorem{theorem}{Theorem}[section]
\newtheorem{proposition}[theorem]{Proposition}
\newtheorem{lemma}[theorem]{Lemma}
\newtheorem{corollary}[theorem]{Corollary}
\newtheorem{remark}[theorem]{Remark}

\newcommand{\pl}{\operatorname{pl}}
\newcommand{\rev}[1]{\overline{#1}}
\newcommand{\ellT}{\ell_T}

\newcommand{\eps}{\varepsilon}

\title{Palindromic Length in Free Groups: Reflections, Noncrossing Matchings, and Catalan Forms}
\author{Junjie Liao\\ \texttt{2230501006@cnu.edu.cn}}
\date{}

\begin{document}
\maketitle

\begin{abstract}
Let $F=F(X)$ be a free group of finite rank, with palindromic length taken
with respect to the fixed basis $X$.  We embed $F$ as the index-two subgroup
of the universal Coxeter group
$W=F\rtimes_\theta\langle t\mid t^2=1\rangle$, where
$\theta(x)=x^{-1}$ for $x\in X$, and prove
\[
   \pl(g)=\min\{\ellT(g),\ellT(gt)\}.
\]
Dyer's deletion theorem then identifies reflection length with the minimum
number of unmatched positions in a noncrossing equal-label partial matching
on a reduced Coxeter word.  This gives an $O(n^3)$-time,
$O(n^2)$-space algorithm for palindromic length, together with recovery of an
optimal palindromic factorization.

The matching model also gives a structural characterization.  For every
ordered full binary tree with $k$ leaves we define a literal word template
whose leaves are palindromes and whose internal vertices carry arbitrary
words.  A reduced word $w$ represents an element of palindromic length at most $k$
if and only if $w$ is
a literal instance of one of these templates.  Hence the $C_{k-1}$ ordered
binary-tree shapes give a complete finite family for each fixed $k$.  For
$k=4$ the five templates are exactly the five forms proposed by Frid, proving
the completeness of that list.  A companion Lean 4 development verifies the four-palindrome
classification end to end for every finite rank, including the ordinary
reduced-word formulation and the literal five-form conclusion.
\end{abstract}



\section{Introduction}

Let $F=F(X)$ be a free group of finite rank, with a basis $X$ fixed
throughout.  A reduced word in $X^{\pm1}$ is a \emph{palindrome} if it reads
the same from left to right and from right to left.  The palindromic length
$\pl(g)$ of $g\in F$ is the least number of palindromes whose product is
$g$.  Unlike palindromic factorization in a free monoid, this is intrinsically
a group-theoretic problem: cancellation between successive factors may create
palindromes which are not visible as subwords of the reduced representative.

Bardakov, Shpilrain and Tolstykh proved that every nonabelian finitely
generated free group has infinite palindromic width \cite{BST}.  In the same
paper they asked whether, given $g\in F$, its palindromic length can be
computed; this is their Problem~2 \cite[Problem~2]{BST}.  The question is
also recorded as Problem~16.9 in the Kourovka Notebook and is still listed in
the current 21st edition \cite[Problem~16.9]{Kourovka}.  More recently, Frid
gave complete descriptions of the reduced words of palindromic length at most
two and at most three \cite[Theorems~1 and~2]{Frid}.  She left the
corresponding problem for larger lengths open, proposed five forms for four
palindromes, and suggested a Catalan--Dyck pattern \cite[Section~5]{Frid}.

There are two earlier observations which point toward Coxeter groups.  If
$\theta\in\operatorname{Aut}(F)$ inverts every basis element, then
Bardakov--Shpilrain--Tolstykh observed that a reduced word $p$ is a palindrome
if and only if $\theta(p)=p^{-1}$ \cite[Lemma~1.4]{BST}.  In the rank-two once-punctured-torus setting,
Kim--Koberda--Lee--Ohshika--Tan--Gao used an index-two embedding of $F_2$ in
the universal Coxeter group on three generators and showed that their
palindromic elements are precisely products of two involutions, one of them
simple \cite[Lemma~4.1]{KimEtAl}.  For a fixed two-element basis, the same
calculation gives the corresponding fixed-basis statement.  Universal Coxeter
groups also occur in the study of palindromic automorphisms; see
Piggott--Ruane \cite{PiggottRuane}.

We use this involutive viewpoint to pass from individual palindromes to the
entire palindromic metric.  Form
\[
        W=F\rtimes_{\theta}\langle t\mid t^2=1\rangle.
\]
Writing $s_0=t$ and $s_x=xt$ for $x\in X$ gives
\[
        W\cong \mathop{\ast}_{\{0\}\sqcup X} C_2.
\]
Let $T$ be the set of reflections of this Coxeter group.  We prove the exact
identity
\begin{equation}\label{eq:intro-length}
        \pl(g)=\min\{\ell_T(g),\ell_T(gt)\}.
\end{equation}
Thus palindromic length is obtained by measuring two adjacent lifts of $g$ in
an index-two Coxeter extension.

Dyer's deletion theorem identifies reflection length with the minimum number
of simple reflections which must be deleted from a fixed reduced Coxeter
expression in order that the remaining subword represent the identity
\cite{Dyer}.  In a universal Coxeter group, an identity subword is equivalent
to a noncrossing perfect matching of its positions by equal labels.  Hence
$\ell_T$ is the minimum number of unmatched positions in a noncrossing
equal-label partial matching.  The interval recurrence for this statistic has
$O(n^2)$ states and $O(n)$ transitions per state.  We consequently obtain an
$O(n^3)$-time, $O(n^2)$-space algorithm for palindromic length, together with
an optimal factorization.  In particular, \eqref{eq:intro-length} and the
matching recurrence answer the algorithmic problem above.

The matching model retains considerably more information than the length
alone.  Every matched arc in a reduced universal-Coxeter word surrounds an
unmatched position.  The leaf intervals enclosed by the arcs therefore form
a laminar family.  After grouping arcs with the same leaf interval and
refining vertices of valence greater than two, the original Coxeter word is
recovered letter for letter as the contour of an ordered full binary tree.
A two-state transducer from the Coxeter extension back to $F$ preserves these
block boundaries literally.  Odd Coxeter-palindrome leaf blocks project to
free-group palindromes, while an internal label enclosing $r$ leaves returns
as its inverse when $r$ is even and as its reversal when $r$ is odd.

This gives the second main result.  For every ordered full binary tree
$\mathcal T$ with $k$ leaves we define a Catalan template
$\mathcal C(\mathcal T)$.  We prove that a reduced word $w$ represents an
element of palindromic length at most $k$ if and only if $w$ is a literal
instance of $\mathcal C(\mathcal T)$ for some such tree.  Consequently the
$C_{k-1}$ ordered tree shapes form a complete finite family of reduced-word
templates.  For $k=2$ and $k=3$ these specialize to Frid's theorems; for
$k=4$ they give exactly the five forms proposed there.  Thus the same argument
establishes the completeness of the four-palindrome list and proves the
Catalan pattern for arbitrary fixed $k$.

\paragraph{Structure of the proof.}
The proof is organized through the following sequence of representations.
Starting from the reduced free word $w$ of an element $g$, we pass through:
\[
\begin{aligned}
\boxed{w\text{ in }F}
&\longrightarrow
\boxed{gt^{\eps}\text{ in }W}
\longrightarrow
\boxed{\text{reduced Coxeter word }v}\\
&\longrightarrow
\boxed{\text{noncrossing matching}}
\longrightarrow
\boxed{\text{binary contour}}
\longrightarrow
\boxed{\text{Catalan template}}.
\end{aligned}
\]
Here $\eps\in\{0,1\}$ is chosen according to the parity of the desired
palindromic bound.  The first arrow converts palindromes into reflections.
The second and third arrows use universal-Coxeter normal form and Dyer's
deletion theorem.  The unmatched positions of the matching become the
leaves of the contour tree.  A two-state projection then returns from the
Coxeter word to the original reduced free word without introducing hidden
cancellation.  Finally, the Catalan template itself can be expanded back into
the same number of palindromes.  Thus the construction is reversible at the
level needed for the main theorem.

For $k=4$, the argument specializes to:
\[
\pl(g)\le4
\Longrightarrow
m\in\{0,2,4\}
\Longrightarrow
\text{a four-leaf Catalan template}
\Longrightarrow
\text{one of five tree shapes}.
\]
The parity restriction $m\in\{0,2,4\}$ will be proved explicitly below; the
five tree shapes are the five ordered full binary trees with four leaves.

\paragraph{Lean verification.}
A companion Lean 4 development, included with the arXiv submission as
ancillary material, verifies the four-palindrome classification end to end.
Its final theorem is stated for every finite rank $n$ and every reduced word
over the basis of \texttt{FreeGroup (Fin n)}; the left-hand side is ordinary
palindromic length in the free group, and the right-hand side is literal
membership in one of the five displayed forms.  The development also checks
the palindrome/reflection bridge, the universal-Coxeter reflection/matching
equivalence, the matching-to-Catalan construction, the projection interfaces,
and the literal Catalan semantics used in the proof.  It contains no proof
placeholders or additional mathematical axiom declarations.  The new
mathematical arguments are given in full below; the principal external input
is Dyer's deletion theorem, together with standard normal-form facts for free
products.  We do not claim formal verification of the asymptotic complexity
analysis of the algorithm.

\subsection*{Notation and background}
We recall the terminology and conventions needed for the proof.  This also
fixes the distinction between equality of group elements and literal equality
of reduced words, which is essential in Sections~5--7.

\begin{itemize}[leftmargin=2em]
\item A \emph{word} is a finite string in symbols $x^{\pm1}$ with $x\in X$.
Two neighboring letters $xx^{-1}$ or $x^{-1}x$ may be cancelled.  A word is
\emph{freely reduced} if no such adjacent cancellation is possible.  Every
element of a free group has a unique freely reduced representative.

\item A \emph{palindrome} is a reduced word that is unchanged by reversal.
For example, $ab^{-1}a$ is a palindrome, whereas $ab^{-1}a^{-1}$ is not.
The product of two palindromes is a group product: after concatenation,
letters at the boundary may cancel.  This is why palindromic length is not
just a string-partition problem.

\item A \emph{universal Coxeter group} is a free product of copies of $C_2$.
Its generators $s_i$ satisfy only $s_i^2=1$.  Thus a Coxeter word is reduced
exactly when no two adjacent labels are equal.

\item A \emph{reflection} is a conjugate $usu^{-1}$ of a Coxeter generator.
Reflection length $\ell_T(h)$ is the smallest number of reflections whose
product is $h$.

\item A \emph{noncrossing matching} pairs some positions of a word by arcs,
with equal labels at the two ends and with no two arcs crossing.  Positions
that are not endpoints of arcs are called \emph{unmatched}.

\item An \emph{ordered full binary tree} is a rooted tree in which every
internal vertex has exactly two ordered children.  Such trees with $k$
leaves are counted by the Catalan number $C_{k-1}$.
\end{itemize}

The proof rests on the following correspondences:
\[
\boxed{\text{palindrome factors}}
\longleftrightarrow
\boxed{\text{reflections}}
\longleftrightarrow
\boxed{\text{unmatched positions}}
\longleftrightarrow
\boxed{\text{tree leaves}}.
\]
Every technical section makes one of these equivalences precise.

Saarela's comparison of palindromic length in free monoids and free groups
provides useful background on the effect of cancellation \cite{Saarela}.
The role of the Coxeter extension here is different from the earlier uses
cited above: the crucial step is the metric identity
\eqref{eq:intro-length}, followed by the deletion/matching interpretation and
the resulting contour normal form.

The paper is organized as follows.  Section~2 establishes the
palindrome--reflection correspondence and the exact length identity.
Section~3 develops the noncrossing matching model, its parity property, and
the cubic-time algorithm.  Section~4 proves the literal contour
decomposition.  Section~5 constructs the two-state projection and explains
how a reduced Coxeter word recovers the original free word.  Section~6 proves
the Catalan template characterization in both directions.  Section~7 treats
the four-palindrome specialization and explains why exactly five forms
occur.  Section~8 gives an example, a proof-chain summary, and further
remarks.

\section{Palindromes and reflections}

\paragraph{Overview.}
A product of palindromes is difficult to read directly from a reduced free
word because neighboring factors may cancel.  The purpose of this section is
to replace each palindrome by a reflection in a slightly larger group.  In
that larger group the relevant length is reflection length, for which Coxeter
theory provides a deletion theorem.  The only new symbol is an involution
$t$; it records a parity state and does not add any new freedom to the free
group element itself.

Let $X^{\pm1}=X\sqcup X^{-1}$.  If
$u=a_1\cdots a_m$ is a word over $X^{\pm1}$, write
\[
        \rev{u}=a_m\cdots a_1
\]
for the reversal of $u$.  Thus $u^{-1}=a_m^{-1}\cdots a_1^{-1}$, and
reversal should not be confused with inversion.  A reduced word $u$ is a
palindrome if $u=\rev{u}$.

For example, if $u=ab^{-1}c$, then
\[
   \rev u=cb^{-1}a,
   \qquad
   u^{-1}=c^{-1}ba^{-1}.
\]
Reversal changes the order but not the signs; inversion changes both.  This
distinction is responsible later for the two different return operations in
Catalan templates.

Let $\theta$ be the automorphism of $F$ defined by
$\theta(x)=x^{-1}$ for every $x\in X$.  On words,
\[
        \theta(u)=\rev{u}^{\,-1}.
\]
Consider
\[
        W=F\rtimes_\theta \langle t\mid t^2=1\rangle,
        \qquad tgt=\theta(g) \quad(g\in F).
\]
For $x\in X$, put $s_x=xt$ and put $s_0=t$.  Then every $s_i$ is an
involution and
\[
        W=\langle s_0,(s_x)_{x\in X}\mid s_i^2=1\rangle
        \cong \mathop{\ast}_{\{0\}\sqcup X} C_2.
\]
Conversely, $x=s_xs_0$, so the even subgroup is naturally identified with
$F$.

\paragraph{Coxeter encoding of a reduced free word.}

The conversion can be performed mechanically.  Replace each positive letter
$x$ by $s_xs_0$ and each negative letter $x^{-1}$ by $s_0s_x$, concatenate,
and then cancel adjacent equal Coxeter generators.  Because the only Coxeter
relations are $s_i^2=1$, this gives the unique reduced Coxeter representative.
If one wants a representative of $gt$ instead of $g$, append one more $s_0$
before the final Coxeter reduction.

We denote this encoding by $\kappa_\eps$, where
$\eps\in\{0,1\}$ records whether we encode $g$ or $gt$.  Section~5 proves
that the two-state projection recovers the unique reduced free word from
$\kappa_\eps(w)$; thus no information about the original word is lost.

For example, if $w=ab^{-1}a$, then
\[
 (s_as_0)(s_0s_b)(s_as_0)
   \;\longrightarrow\;
 s_as_bs_as_0.
\]
Later the two-state projection will read this reduced Coxeter word back as
$a b^{-1}a$.  The general construction is the same.

Let $S=\{s_0\}\cup\{s_x:x\in X\}$ and
\[
        T=\{wsw^{-1}:w\in W,\ s\in S\}
\]
be the set of reflections.  We write $\ell_T$ for word length with respect
to $T$.

The equivalence between palindromes and elements inverted by $\theta$ is
\cite[Lemma~1.4]{BST}.  A closely related rank-two Coxeter-extension
formulation appears in \cite[Lemma~4.1]{KimEtAl}.  We record the version
needed here because its coset formulation will be used repeatedly in the
length argument.

The next lemma is the conceptual hinge of the paper.  It says that one
palindrome in the free group is exactly one reflection after adjoining the
parity letter $t$.

\begin{lemma}[Palindrome--reflection correspondence]\label{lem:pal-ref}
The map
\[
        p\longmapsto pt
\]
is a bijection from the set of palindromes in $F$ (including the identity)
to the set $T$ of reflections in $W$.
\end{lemma}

\begin{proof}
For $p\in F$,
\[
        (pt)^2=p\theta(p).
\]
Hence $(pt)^2=1$ if and only if $\theta(p)=p^{-1}$, which is equivalent to
the reduced word for $p$ being a palindrome by \cite[Lemma~1.4]{BST}.  Since $pt$ lies in the odd
coset, it is nontrivial.  A nontrivial finite-order element of a free
product is conjugate into a free factor; as every free factor of $W$ has
order two, every nontrivial involution is a reflection.  Thus a palindrome
$p$ gives a reflection $pt$.

Conversely, every reflection is an involution and lies in the odd coset:
the parity homomorphism $W\to C_2$ sends each simple reflection to the
nontrivial element and is invariant under conjugation.  Hence a reflection
has a unique form $qt$ with $q\in F$.  From $(qt)^2=1$ we obtain
$\theta(q)=q^{-1}$, so the reduced word for $q$ is a palindrome.  Uniqueness
of the normal form $qt$ proves bijectivity.
\end{proof}

The following form of the length correspondence is useful because it keeps
track of parity.  The identity palindrome is important here: it lets us pad a
factorization by one factor at a time on the free-group side, while a pair
$(t)(t)$ pads a reflection factorization by two factors on the Coxeter side.

\begin{proposition}\label{prop:bounded}
For every $g\in F$ and every integer $k\ge 0$,
\[
        \pl(g)\le k
        \quad\Longleftrightarrow\quad
        \ell_T\bigl(gt^{k\bmod2}\bigr)\le k.
\]
\end{proposition}

\begin{proof}
Suppose first that $g=p_1\cdots p_r$ with $r\le k$ and each $p_i$ a
palindrome.  Inserting identity palindromes if necessary, assume $r=k$.
Since $\theta$ preserves palindromes,
\[
 gt^{k\bmod2}
   =(p_1t)(\theta(p_2)t)(p_3t)(\theta(p_4)t)\cdots,
\]
and every factor on the right is a reflection by
Lemma~\ref{lem:pal-ref}.  Thus the reflection length is at most $k$.

Conversely, put $h=gt^{k\bmod2}$ and suppose $\ell_T(h)=m\le k$.
Every reflection lies in the odd coset $Ft$, so $m\equiv k\pmod2$.
After inserting $(t)(t)$ pairs we may therefore write $h$ as a product of
exactly $k$ reflections.  By Lemma~\ref{lem:pal-ref}, write those factors as
$q_it$ with each $q_i$ a palindrome.  Multiplying in the semidirect product
gives
\[
        g=q_1\theta(q_2)q_3\theta(q_4)\cdots,
\]
a product of $k$ palindromes.
\end{proof}

\begin{theorem}\label{thm:length-formula}
For every $g\in F$,
\[
        \boxed{\displaystyle
        \pl(g)=\min\{\ell_T(g),\ell_T(gt)\}.}
\]
\end{theorem}

\begin{proof}
Apply Proposition~\ref{prop:bounded} with $k=\ell_T(g)$ and with
$k=\ell_T(gt)$, noting that $\ell_T(g)$ is even and $\ell_T(gt)$ is odd.
This gives
$\pl(g)\le\min\{\ell_T(g),\ell_T(gt)\}$.  Conversely, if
$k=\pl(g)$, Proposition~\ref{prop:bounded} gives
$\ell_T(gt^{k\bmod2})\le k$, proving the reverse inequality.
\end{proof}

\paragraph{Consequence for the sequel.}
The free-group problem has now been converted into a Coxeter-length problem.
For a bound $k$, we do not need to guess a palindrome factorization directly:
we study the single Coxeter element $gt^{k\bmod2}$ and ask whether its
reflection length is at most $k$.  All later sections work on a reduced
Coxeter word for this element.

\section{Noncrossing matchings and a cubic-time algorithm}

\paragraph{Overview.}
Dyer's theorem asks how many letters of a reduced Coxeter expression must be
deleted so that the surviving subword represents the identity.  In a
universal Coxeter group, identity reduction consists only of cancelling equal
letters in pairs.  Recording these cancellations by arcs above the word turns
the deletion problem into a noncrossing matching problem.  Deleted letters
are precisely the unmatched positions.  This translation is the reason a
group-theoretic length problem becomes a finite dynamic program.

Let
\[
        v=v_1\cdots v_N,\qquad v_i\in S,
\]
be a reduced word in the universal Coxeter group $W$; thus
$v_i\ne v_{i+1}$ for all $i$.  A \emph{noncrossing equal-label partial
matching} on $v$ is a collection of disjoint pairs $(i,j)$, $i<j$, such
that $v_i=v_j$ and no two pairs cross: there are no
$i<k<j<\ell$ with both $(i,j)$ and $(k,\ell)$ in the matching.  Write
$\delta(v)$ for the minimum number of unmatched positions.

For example, consider the reduced Coxeter word
\[
       v=s_as_bs_a.
\]
The first and third positions may be paired, leaving only the middle
$s_b$ unmatched.  Hence this particular matching has one unmatched position,
and in fact $\delta(v)=1$.  The group element represented by $v$ is therefore
a single reflection.  The example illustrates the dictionary used below:
an unmatched letter contributes one reflection factor.

The following elementary parity observation is used later to pass from an
``at most $k$'' matching to a template with exactly $k$ leaves.

\begin{lemma}[Parity of unmatched positions]\label{lem:matching-parity}
If a word of length $N$ carries a partial matching with $d$ unmatched
positions, then
\[
        d\equiv N\pmod2.
\]
\end{lemma}

\begin{proof}
If the matching contains $q$ pairs, then exactly $2q$ positions are matched,
so $N=2q+d$.
\end{proof}

\begin{lemma}\label{lem:identity-matching}
A subword of $v$ represents the identity in $W$ if and only if its selected
positions admit a noncrossing perfect matching by equal labels.
\end{lemma}

\begin{proof}
If the subword reduces to the empty word, record each cancellation of two
adjacent equal letters.  Reading the cancellations in reverse produces a
noncrossing perfect matching of equal letters.  Conversely, a noncrossing
perfect matching contains an innermost matched pair.  In the selected
subword its two letters are adjacent and equal, hence cancel.  Induction on
the number of matched pairs completes the proof.
\end{proof}

\begin{lemma}[Matching-to-reflections]\label{lem:matching-reflections}
If a Coxeter word carries a noncrossing equal-label partial matching with
$d$ unmatched positions, then the represented group element is a product of
at most $d$ reflections.  Given the matching, such a factorization can be
constructed recursively.
\end{lemma}

\begin{proof}
Induct on the length of the word.  If the first position is unmatched, write
the word as $sz$.  The restricted matching on $z$ has $d-1$ unmatched
positions, so by induction $[z]$ is a product of at most $d-1$ reflections.
Since $s$ is itself a reflection, $[sz]$ is a product of at most $d$
reflections.

Otherwise the first position is matched with a later occurrence of the same
letter, and the word has the form
\[
        s\,u\,s\,z.
\]
Noncrossing implies that the matching restricts independently to $u$ and
$z$.  If these restrictions have $d_1$ and $d_2$ unmatched positions, then
$d_1+d_2=d$.  By induction write
\[
        [u]=r_1\cdots r_{d_1},\qquad
        [z]=q_1\cdots q_{d_2},
\]
allowing fewer factors if necessary.  Since reflections are closed under
conjugation,
\[
        [s u s]
        =(s r_1s)\cdots(s r_{d_1}s)
\]
is a product of at most $d_1$ reflections.  Appending the factorization of
$[z]$ gives the required factorization of the whole word.  The same
induction is an explicit reconstruction procedure.
\end{proof}

Dyer's deletion theorem now has a particularly simple consequence.

\begin{proposition}\label{prop:delta-reflection}
For a reduced universal Coxeter word $v$ representing $w\in W$,
\[
        \boxed{\delta(v)=\ell_T(w).}
\]
\end{proposition}

\begin{proof}
By Dyer's theorem \cite{Dyer}, $\ell_T(w)$ is the least number of letters
which must be deleted from the reduced expression $v$ so that the remaining
subword represents the identity.  By Lemma~\ref{lem:identity-matching},
keeping a subword which represents the identity is equivalent to pairing all
of its positions by a noncrossing equal-label matching.  The deleted
positions are exactly the unmatched positions.
\end{proof}

Thus there is no loss of information in replacing reflection length by the
minimum number of unmatched positions.  In particular, a witness matching
simultaneously provides an upper bound and, through Dyer's theorem, certifies
that no shorter reflection factorization exists.

The statistic $\delta$ is computed by a standard interval recurrence.  For
a half-open interval $[i,j)$, let $D(i,j)$ be the minimum number of
unmatched positions among positions $i,i+1,\ldots,j-1$.  Then
\begin{equation}\label{eq:dp}
D(i,j)=\min\left\{
1+D(i+1,j),
\min_{\substack{i<r<j\\v_r=v_i}}
\bigl(D(i+1,r)+D(r+1,j)\bigr)
\right\}.
\end{equation}
Indeed, position $i$ is either left unmatched or is matched with some
position $r$ carrying the same label; noncrossing then separates the two
remaining intervals.

\begin{theorem}[Algorithmic palindromic length]\label{thm:algorithm}
Let $F$ be a finitely generated free group with fixed basis, let $g\in F$,
and let $w$ be the reduced word of $g$, of length $n$.  Then $\pl(g)$ can be
computed in time $O(n^3)$ and space $O(n^2)$.  An optimal palindromic
factorization of $g$ can be recovered within the same asymptotic bounds.
\end{theorem}

\begin{proof}
Embed $F$ in $W$ using $x\mapsto s_xs_0$ and
$x^{-1}\mapsto s_0s_x$, and freely reduce in the Coxeter alphabet.  This
produces reduced Coxeter representatives of $g$ and $gt$, each of length at
most $2n+1$.  Compute $\delta$ for each representative by
\eqref{eq:dp}.  There are $O(n^2)$ intervals and $O(n)$ possible partners
for the left endpoint of an interval, giving time $O(n^3)$ and space
$O(n^2)$.  Proposition~\ref{prop:delta-reflection} and
Theorem~\ref{thm:length-formula} give the exact palindromic length.

Recording a minimizing choice in each dynamic-programming state recovers an
optimal matching and hence a minimum deletion set.  For either chosen lift
$h\in\{g,gt\}$, Lemma~\ref{lem:matching-reflections} constructs a reflection
factorization of $h$ with at most $\delta(v)$ factors.  Since
$\delta(v)=\ell_T(h)$ by Proposition~\ref{prop:delta-reflection}, minimality
forces this factorization to have exactly $\delta(v)$ factors.  Lemma~
\ref{lem:pal-ref} and the proof of Proposition~\ref{prop:bounded} then convert
the relevant minimum reflection factorization into an optimal palindromic
factorization.
\end{proof}

\begin{remark}
The theorem establishes decidability with an explicit polynomial bound; no
claim is made here that the cubic exponent is optimal.
\end{remark}

\section{Contours carried by noncrossing matchings}

\paragraph{Overview.}

The goal of this section is not to introduce a new invariant.  It is to
reorganize one optimal matching into a tree without changing the word.  The
unmatched positions will become leaves.  Matched arcs that surround the same
set of leaves form nested boundary layers.  Removing those layers exposes
smaller independent pieces, which become the children of a vertex.  If a
vertex naturally has more than two children, we insert empty-boundary
vertices until the tree is binary.  These inserted vertices change neither
the Coxeter word nor its later free-word projection.

A useful picture is therefore
\[
\text{outer matched arcs}
\;\bigl[\;\text{left child}\;\text{right child}\;\bigr]\;
\text{mirror outer arcs}.
\]
The technical lemmas below prove that this picture is literal: every position
of the Coxeter word belongs to exactly one boundary block or one child
support, so no letters are lost or duplicated.

The dynamic program only remembers the number of unmatched positions, but
for the classification theorem we need the geometry of an optimal matching.
Think of the unmatched positions as future leaves.  A matched arc records a
piece of boundary wrapped around some consecutive leaves.  Noncrossing means
that these leaf sets are nested or disjoint, so they form a laminar family.
After collapsing chains of arcs with the same leaf set and splitting vertices
with more than two children, the laminar family becomes an ordered full
binary tree.  The point of the technical lemmas below is to prove that this
tree is not merely an abstract summary: its contour reproduces the original
Coxeter word letter for letter.

We now retain more information from a noncrossing matching.  The results of
this section are purely combinatorial.  Throughout this section we assume
that the matching has at least one unmatched position.  The case of zero
unmatched positions is not needed for the contour construction and is handled
separately in the proof of Theorem~\ref{thm:catalan}.

Fix a reduced Coxeter word $v=v_1\cdots v_N$ and a noncrossing equal-label
matching with $m\ge1$ unmatched positions
\[
        u_1<u_2<\cdots<u_m.
\]
For a matched pair $e=(i,j)$ define its \emph{leaf interval}
\[
        I(e)=\{r:i<u_r<j\}\subseteq\{1,\ldots,m\}.
\]

\begin{lemma}\label{lem:arc-leaf}
Every matched pair contains an unmatched position in its interior.  In
particular, $I(e)$ is a nonempty interval of consecutive integers.
\end{lemma}

\begin{proof}
Suppose $(i,j)$ contains no unmatched position.  Every position strictly
between $i$ and $j$ is then perfectly matched inside $(i,j)$; by
noncrossing, no such position can be matched outside the interval.  Hence
the interior subword represents the identity by
Lemma~\ref{lem:identity-matching}.  Since $v_i=v_j$, the entire subword
$v_i\cdots v_j$ also represents the identity.  But it is a nonempty
contiguous subword of a reduced free-product normal form, a contradiction.
\end{proof}

The lemma has a useful picture.  No matched arc is allowed to surround a
region containing only matched positions; otherwise that entire region would
collapse to the identity inside a reduced word.  Hence every arc must
``see'' at least one unmatched position, and those unmatched positions are
exactly what will become the leaves of the tree.

The leaf intervals form a laminar family: two of them are disjoint or one
contains the other.  We now make the resulting interval structure completely
literal at the level of positions of the Coxeter word.

Let
\[
 \mathcal L=\{I(e):e\text{ is a matched pair}\}
 \cup\bigl\{\{1\},\ldots,\{m\}\bigr\}
 \cup\{[1,m]\},
\]
where repeated intervals are identified.  For $J\in\mathcal L$, write
$\mathcal E(J)$ for the set of matched pairs $e$ with $I(e)=J$.

\begin{lemma}[Boundary-chain lemma]\label{lem:chain}
For every $J\in\mathcal L$, the arcs in $\mathcal E(J)$ form a nested chain.
If two consecutive members of this chain are
\[
        (i,j)\supset(k,\ell),
\]
then
\[
        k=i+1,\qquad \ell=j-1.
\]
Thus, if
\[
 (a_1,b_1)\supset\cdots\supset(a_q,b_q)
\]
is the complete chain $\mathcal E(J)$, listed from outside to inside, and
\[
        U_J=v_{a_1}\cdots v_{a_q},
\]
then the corresponding right boundary is literally $\rev{U_J}$.
\end{lemma}

\begin{proof}
Two arcs with the same nonempty leaf interval cannot be disjoint, so
noncrossing forces them to be nested.  Let
$(i,j)\supset(k,\ell)$ be consecutive in the chain and suppose that
$i<p<k$.  The position $p$ is not unmatched, since otherwise it would lie
inside the outer arc but outside the inner one.  Hence $p$ is matched.

If its partner also lies in $(i,k)$, that matched pair contains no unmatched
position, contrary to Lemma~\ref{lem:arc-leaf}.  A partner in $(k,\ell)$ or
outside $(i,j)$ would force a crossing with one of the two given arcs.  The
only remaining possibility is a partner in $(\ell,j)$.  But then the new arc
contains exactly the same unmatched positions as $(i,j)$ and $(k,\ell)$ and
lies strictly between them, contradicting their consecutiveness in the
complete chain.  Hence $k=i+1$.  The same argument on the right gives
$\ell=j-1$.  Since the two endpoints of each matched pair have the same
Coxeter label, the right boundary is the reversal of the left one.
\end{proof}

\paragraph{Interpretation.}
All arcs surrounding exactly the same group of future leaves are forced to
sit concentrically with no unused positions between consecutive arcs.  They
therefore contribute one literal word $U_J$ on the left and the reversed
word $\rev{U_J}$ on the right.  These paired boundary words will become the
labels of internal tree vertices.

We next attach to every $J\in\mathcal L$ a literal interval of positions,
called its \emph{support}.  If $\mathcal E(J)\ne\varnothing$, let
$H_J=[a_1,b_1]$, where $(a_1,b_1)$ is the outermost arc in the complete
chain above.  If $\mathcal E(J)=\varnothing$ and $J=\{r\}$, set
$H_J=\{u_r\}$.  When $m=1$, this singleton is also the root; in the
boundary-free case Lemma~\ref{lem:arc-leaf} forces there to be no matched
positions at all, so $N=1$ and this convention agrees with $H_J=[1,N]$.
In the remaining boundary-free root case, namely $J=[1,m]$ with $m>1$, set
$H_J=[1,N]$.  If $q=|\mathcal E(J)|$, define the \emph{core} of $J$ by
removing the $q$ left and $q$ right boundary positions from $H_J$.

\begin{lemma}[Support-partition lemma]\label{lem:support-partition}
The following hold.
\begin{enumerate}[label=(\roman*)]
\item The support of the root interval $[1,m]$ is $[1,N]$.
\item If $J=\{r\}$ is a singleton, then the core of $J$ consists of the
single position $u_r$.
\item Let $|J|>1$, and let
\[
        J_1,\ldots,J_d
\]
be the maximal proper members of $\mathcal L$ contained in $J$, listed from
left to right.  Then $d\ge2$, and
\[
        H_{J_1},\ldots,H_{J_d}
\]
are pairwise disjoint consecutive intervals whose union is exactly the core
of $J$.
\end{enumerate}
\end{lemma}

\begin{proof}
For (i), suppose first that the root has a nonempty boundary chain, with
outermost arc $(a,b)$.  Every unmatched position lies in $(a,b)$.  If there
were a position $p<a$, then $p$ would have to be matched.  Its partner cannot
also lie to the left of $a$, because that arc would contain no unmatched
position; it cannot lie in $(a,b)$, because it would cross $(a,b)$; and if it
lay to the right of $b$, the resulting arc would contain all unmatched
positions and strictly contain the outermost root arc.  All three
possibilities are impossible.  Thus $a=1$, and similarly $b=N$.  If the root
has no boundary chain, its support was defined to be $[1,N]$.

For (ii), remove the complete chain of arcs whose leaf interval is
$\{r\}$.  Any remaining position in the resulting core other than $u_r$
would have to be matched.  Its arc has a nonempty leaf interval by
Lemma~\ref{lem:arc-leaf}; since it lies inside the innermost boundary arc, or
inside the root when the boundary chain is empty, that interval is contained
in $\{r\}$ and hence equals $\{r\}$.  This contradicts completeness of the
removed chain.  Thus only $u_r$ remains.

For (iii), laminarity implies that the maximal proper members of $J$ are
pairwise disjoint.  They cover the leaf set $J$: every $r\in J$ belongs to
the singleton $\{r\}\in\mathcal L$, and in the finite inclusion poset this
singleton lies below a maximal proper member of $J$.  Since $|J|>1$, one
maximal proper member cannot cover all of $J$, so $d\ge2$.

It remains to prove the literal statement about positions.  Let $p$ be any
position in the core of $J$.  If $p=u_r$ is unmatched, then $r$ lies in a
unique maximal proper member $J_s$, and by the definition of support we have
$p\in H_{J_s}$.  Suppose instead that $p$ is an endpoint of a matched arc
$e$.  The boundary positions belonging to $\mathcal E(J)$ have already been
removed, so $I(e)\ne J$.  Noncrossing implies that $e$ lies inside the
innermost boundary arc of $J$ when such an arc exists; for the root with
empty boundary this is automatic.  Hence $I(e)$ is a proper member of
$\mathcal L$ contained in $J$.  It is therefore contained in a unique
maximal proper member $J_s$.  If $I(e)=J_s$, then $e$ lies in the complete
boundary chain defining $H_{J_s}$; if $I(e)\subsetneq J_s$, noncrossing
forces $e$ to lie inside the outermost arc defining $H_{J_s}$.  Thus again
$p\in H_{J_s}$.

We have proved that the child supports cover the core.  Two distinct child
supports cannot overlap.  Indeed, if both are bounded by arcs, an overlap of
their support intervals would force the arcs to cross or one to contain the
other; the latter would force containment of the corresponding nonempty
leaf intervals, contradicting maximality and disjointness.  If one support
is a singleton unmatched position, lying inside the other support would put
that leaf into the other leaf interval.  Hence the supports are disjoint.
Their left-to-right order is the order of the leaf intervals, and since they
cover a single interval, they are consecutive with no gaps.
\end{proof}

\paragraph{Interpretation.}
After the left and right boundary chains of a vertex are removed, the remaining middle block is determined exactly: the core is exactly the consecutive union
of the child supports.  Thus every Coxeter position belongs either to a
vertex boundary or to one child, and no letter is lost when the matching is
turned into a tree.

The two lemmas above give the contour recursion without any appeal to a
picture.

At this point the combinatorics can be summarized as follows.  The unmatched
positions give the leaves, complete chains of equal leaf interval give the
left/right boundary words of vertices, and the support-partition lemma says
that the children fill the core with no gaps.  Refining an ordered vertex of
valence $d>2$ by inserting empty-boundary binary vertices changes no word.
We therefore obtain a genuine full binary contour.

\begin{proposition}[Literal Coxeter contour decomposition]\label{prop:contour}
Assume $m\ge1$.  The reduced Coxeter word $v$ admits an ordered rooted-tree
decomposition with leaf set $\{1,\ldots,m\}$ such that:
\begin{enumerate}[label=(\roman*)]
\item the leaf corresponding to $r$ is the odd Coxeter palindrome
\[
        Q_r=U_{\{r\}}\,v_{u_r}\,\rev{U_{\{r\}}};
\]
\item if $J$ is a nonsingleton vertex with ordered children
$J_1,\ldots,J_d$, then the subword on $H_J$ is literally
\[
        U_J\,V_{J_1}\cdots V_{J_d}\,\rev{U_J},
\]
where $V_{J_s}$ is the subword on $H_{J_s}$.
\end{enumerate}
Every internal vertex has at least two children.  Replacing each vertex with
$d>2$ children by any ordered full binary refinement and assigning the empty
word to each inserted internal vertex gives an ordered full binary tree with
$m$ leaves and leaves the Coxeter word unchanged letter for letter.
\end{proposition}

\begin{proof}
Order the augmented laminar family $\mathcal L$ by inclusion.  Its leaves are
the singletons.  By Lemma~\ref{lem:support-partition}(iii), the maximal
proper members below every nonsingleton $J$ form its ordered list of at
least two children.

For a singleton, Lemma~\ref{lem:support-partition}(ii) and the
Boundary-chain lemma give exactly the word in (i), which has odd length and
is a Coxeter palindrome.  For a nonsingleton $J$, the Boundary-chain lemma
identifies the literal prefix and suffix of $H_J$ with $U_J$ and $\rev{U_J}$,
while Lemma~\ref{lem:support-partition}(iii) says that the remaining core is
precisely the consecutive concatenation of the child supports.  This proves
(ii).  Starting from the root, whose support is the whole word by
Lemma~\ref{lem:support-partition}(i), induction down the inclusion tree
therefore reconstructs every position of $v$ exactly once and yields $v$
itself as the contour word.

Finally, a $d$-fold ordered concatenation of child blocks may be parenthesized
by an ordered full binary tree.  Every newly inserted internal vertex is
given empty boundary word, so its contour operation is simply concatenation.
Thus binary refinement changes neither a letter nor a block boundary of the
original Coxeter word.
\end{proof}

\begin{remark}
Optimality of the matching is not used in Proposition~\ref{prop:contour}.
For a reduced universal-Coxeter word, Lemma~\ref{lem:arc-leaf} forces the
required unmatched leaf inside every matched arc for \emph{every}
noncrossing equal-label partial matching.  Optimality enters only later,
when the number of leaves is identified with reflection length.
\end{remark}

\section{A two-state projection to the free group}

\paragraph{Overview.}
The contour of Section~4 still lives in the Coxeter alphabet.  We now read it
back into the free-group alphabet.  The only information needed is whether an
even or odd number of Coxeter letters has already been read.  This is why a
two-state transducer is sufficient.  The key point is stronger than equality
in the group: for a reduced Coxeter word, the projected free word is already
freely reduced, so the block boundaries of the contour survive literally.

The embedding $F\le W$ can be inverted on normal forms by a two-state
transducer.  The explicit form of this transducer is what preserves the
literal block structure obtained in Proposition~\ref{prop:contour}.

Let the state be $\eps\in\{0,1\}$.  Reading one Coxeter letter changes the
state by $1$ modulo $2$.  The letter $s_0$ produces no free letter.  If
$s_x$ with $x\in X$ is read in state $0$, output $x$; if it is read in
state $1$, output $x^{-1}$.  Denote by $\Phi_\eps(U)$ the output obtained by
reading a Coxeter word $U$ from initial state $\eps$.

The rule can be displayed in one table:
\[
\begin{array}{c|cc}
\text{Coxeter letter} & \text{state }0 & \text{state }1\\ \hline
s_0 & \varnothing & \varnothing\\
s_x & x & x^{-1}
\end{array}
\]
and the state flips after every Coxeter letter.  For the earlier example
$s_as_bs_as_0$, starting in state $0$ produces
\[
        a\,b^{-1}\,a,
\]
exactly the original reduced free word.

\begin{lemma}[Literal projection]\label{lem:projection}
Let $v$ be a reduced Coxeter word.
\begin{enumerate}[label=(\roman*)]
\item For either initial state $\eps$, the word $\Phi_\eps(v)$ is freely
reduced.
\item If $[v]=gt^{|v|\bmod2}$ in semidirect-product normal form, then
$\Phi_0(v)$ is the reduced word representing $g$.
\item If $v=U_1\cdots U_r$ is any decomposition into contiguous blocks and
the terminal state of each block is passed to the next, then the block
outputs concatenate literally to $\Phi_\eps(v)$.  In particular no hidden
free cancellation occurs at a block boundary.
\end{enumerate}
\end{lemma}

\begin{proof}
Consider two consecutive output letters.  They come from consecutive
nonzero Coxeter letters, with either no $s_0$ or exactly one $s_0$ between
them.  In the first case their signs are opposite; if they were inverse
free letters, the two Coxeter labels would be equal, contradicting reducedness
of $v$.  In the second case their signs are equal, so they cannot be
inverses.  This proves (i), for either initial state.

Statement (ii) is obtained by maintaining the semidirect-product normal form
while the word is read.  Statement (iii) is the associativity of the same
finite-state reading process together with (i).
\end{proof}

\begin{corollary}[Recovery of the reduced free word]\label{cor:recovery}
Let $w$ be the reduced free word representing $g\in F$, let
$\eps\in\{0,1\}$, and let $v$ be any reduced Coxeter word representing
$gt^\eps$.  Then
\[
        |v|\equiv\eps\pmod2,
        \qquad
        \Phi_0(v)=w.
\]
Thus the projection does not depend on which reduced Coxeter representative
is chosen.
\end{corollary}

\begin{proof}
Every Coxeter generator lies in the odd coset, so the parity of $|v|$ is the
coset parity $\eps$.  Lemma~\ref{lem:projection}(ii) then says that
$\Phi_0(v)$ is the reduced word representing $g$, hence it is $w$.
\end{proof}

\begin{lemma}[Projected leaves]\label{lem:leaf-projection}
If $Q$ is an odd Coxeter palindrome, then $\Phi_\eps(Q)$ is a free-group
palindrome for either initial state $\eps$.
\end{lemma}

\begin{proof}
Write $|Q|=2r+1$.  Mirror positions of $Q$ have the same Coxeter label and
the same parity of position.  Hence, whenever that label is nonzero, the
two positions output the same free generator with the same sign.  Positions
labelled $s_0$ produce no output.  The output is therefore invariant under
reversal.
\end{proof}

The next phenomenon explains the mixture of inversion and reversal in the
final formulas.  An internal boundary word is read once on the way into a
subtree and then again, backwards, on the way out.  Whether the signs also
flip depends only on the parity of the material enclosed by that boundary.
For a free word $A$, define
\[
\rho_r(A)=
\begin{cases}
A^{-1},& r\text{ even},\\
\rev{A},& r\text{ odd}.
\end{cases}
\]

\begin{lemma}[Return parity]\label{lem:return}
Suppose a Coxeter block has the form $U V \rev U$, and let
$A=\Phi_\eps(U)$.  If the state is propagated through $U$ and $V$, then the
output on the returning block $\rev U$ is
\[
\begin{cases}
A^{-1},& |V|\text{ even},\\
\rev{A},& |V|\text{ odd}.
\end{cases}
\]
This holds for either initial state $\eps$.
\end{lemma}

\begin{proof}
The return block reads the Coxeter letters of $U$ in reverse order.  Thus
the output order is reversed.  Comparing a letter of $U$ with its mirror in
$\rev U$, the parity of the intervening number of state changes differs by
$|V|+1$.  Hence the free sign is reversed when $|V|$ is even and preserved
when $|V|$ is odd.  Reversal together with sign reversal is $A^{-1}$;
reversal without sign change is $\rev A$.
\end{proof}

\begin{lemma}[Parity of a contour subtree]\label{lem:subtree-parity}
If a contour subtree has Coxeter word $V$ and exactly $r$ leaves, then
\begin{equation}\label{eq:leaf-parity}
        |V|\equiv r\pmod2.
\end{equation}
\end{lemma}

\begin{proof}
Each internal boundary word occurs twice, once on the way into its subtree
and once in reverse order on the way out, and therefore contributes an even
number of Coxeter letters.  Each leaf block is an odd Coxeter palindrome and
therefore contributes an odd number of letters.  Modulo two, the total
length is consequently the number of leaves.
\end{proof}

Thus the return operation at a vertex depends only on the number of leaves
below that vertex.

Before packaging the projection lemmas, we define the word templates to
which contours will project.  Let $\mathcal T$ be an ordered full binary
tree.  Write $\lambda(v)$ for the number of leaves below a vertex $v$.
Attach an arbitrary free-word variable $A_v$ to every internal vertex and a
palindrome variable $P_v$ to every leaf.  Define
$\mathcal C(\mathcal T)$ recursively by
\[
        \mathcal C(v)=P_v
\]
for a leaf, and
\begin{equation}\label{eq:catalan-template}
\mathcal C(v)
   =A_v\,\mathcal C(v_L)\,\mathcal C(v_R)\,
      \rho_{\lambda(v)}(A_v)
\end{equation}
for an internal vertex.

A reduced free word $w$ is a \emph{literal instance} of
$\mathcal C(\mathcal T)$ if the variables can be replaced by words, with
the leaf variables replaced by palindromes, so that the displayed
concatenation is exactly $w$ as a word.  Empty words are allowed.  In
particular, no cancellation between template blocks is part of the
definition.  This word ``literal'' is important: the theorem classifies the
actual reduced word, not merely the group element represented by an
expression that later reduces.  The arbitrary internal words record the
cancellation that occurred before the Coxeter contour was projected, while
the leaf blocks are the genuine palindrome pieces.

The preceding local facts can now be packaged into a single statement.
This formulation makes the state propagation in later recursive arguments
explicit.

\begin{proposition}[Projection of reduced contours]\label{prop:project-contour}
Let $\mathcal T$ be an ordered full binary tree.  Attach a Coxeter word
$U_z$ to every internal vertex $z$ and an odd Coxeter palindrome $Q_z$ to
every leaf.  Define recursively
\[
        V_z=Q_z
\]
at a leaf and
\[
        V_z=U_z\,V_{z_L}\,V_{z_R}\,\rev{U_z}
\]
at an internal vertex.  Assume that the root contour $V_{\mathrm{root}}$ is
a reduced Coxeter word (as it is in the application coming from
Proposition~\ref{prop:contour}).  Then for every vertex $z$ and every initial
state $\eps\in\{0,1\}$, the word $\Phi_\eps(V_z)$ is literally an instance of
the Catalan template associated with the subtree rooted at $z$.  More
precisely, the variable attached to an internal vertex $y$ is the output
$\Phi_{\eps_y}(U_y)$ in the actual state $\eps_y$ in which that block is
entered, and each leaf variable is the projection of its leaf block in its
actual entry state.  The projected root template instance is freely reduced.
\end{proposition}

\begin{proof}
We argue by structural induction, simultaneously for both possible entry
states.  At a leaf, Lemma~\ref{lem:leaf-projection} says that
$\Phi_\eps(Q_z)$ is a palindrome, so the assertion is exactly the one-leaf
template.

Let $z$ be internal and write
\[
        V_z=U_zV_LV_R\rev{U_z}.
\]
Read $U_z$ in the actual entry state $\eps$ and put
$A_z=\Phi_\eps(U_z)$.  Pass the resulting state to $V_L$ and then the
terminal state of $V_L$ to $V_R$.  The induction hypothesis applies to both
children because it was stated for an arbitrary entry state, so their
outputs are literal instances of their respective subtree templates.

The middle block $V_LV_R$ has
\[
        |V_LV_R|
          \equiv \lambda(z_L)+\lambda(z_R)
          =\lambda(z)\pmod2,
\]
where $\lambda(y)$ denotes the number of leaves below $y$; this is Lemma~\ref{lem:subtree-parity}.  Lemma~\ref{lem:return} therefore shows that the
returning copy $\rev{U_z}$ projects to
\[
        \rho_{\lambda(z)}(A_z).
\]
Since the two-state transducer concatenates the outputs of consecutive
blocks literally, the complete output is
\[
        A_z\,
        \mathcal C(z_L)\,
        \mathcal C(z_R)\,
        \rho_{\lambda(z)}(A_z),
\]
with the displayed child templates understood under the assignments supplied
by induction.  This is exactly the Catalan recursion at $z$.

Finally, the root contour is reduced by hypothesis.  Every subtree contour
is a contiguous subword of it and is therefore reduced as well.
Lemma~\ref{lem:projection}(i) shows in particular that the root projection in
either entry state is freely reduced.
\end{proof}

\section{Catalan forms}

\paragraph{Overview.}
Sections~2--5 prove the structural implication from bounded palindromic
length to a Catalan contour.  We now prove the converse directly at word level.  The only algebraic mechanism is
that an arbitrary boundary word $A$ can be wrapped around a product of
palindromes without increasing the number of palindrome factors, provided
the returning copy of $A$ is chosen according to the parity rule $\rho_r$.
This is exactly the rule already forced on us by the two-state projection.

We first record the closure property needed for the converse.

\begin{lemma}\label{lem:closure}
Let $A$ be any free word.  If $h$ is a product of $r$ palindromes, then
\[
        A h\rho_r(A)
\]
is again a product of $r$ palindromes.
\end{lemma}

\begin{proof}
If $r=0$, then $h=1$ and $\rho_0(A)=A^{-1}$, so the assertion is
immediate.  Assume $r\ge1$.  Write $h=p_1\cdots p_r$, with every $p_i$ a
palindrome, and let $a\in F$ be the element represented by the word $A$.  The word $\rev A$ represents
$\theta(a)^{-1}$.  Hence the two alternating types of factors
\[
        A p_i\rev A,
        \qquad
        (\rev A)^{-1}p_iA^{-1}
\]
represent, respectively,
\[
        a p_i\theta(a)^{-1},
        \qquad
        \theta(a)p_i a^{-1}.
\]
Both are palindrome elements.  Indeed, since
$\theta(p_i)=p_i^{-1}$,
\[
\theta\!\left(a p_i\theta(a)^{-1}\right)
   =\theta(a)p_i^{-1}a^{-1}
   =\left(a p_i\theta(a)^{-1}\right)^{-1},
\]
and similarly
\[
\theta\!\left(\theta(a)p_i a^{-1}\right)
   =a p_i^{-1}\theta(a)^{-1}
   =\left(\theta(a)p_i a^{-1}\right)^{-1}.
\]
By the palindrome criterion from Section~2, the reduced representative of
each of these elements is therefore a palindrome.  Notice that this argument
does not assume that the displayed concatenations are themselves freely
reduced.

Now alternate these two types of palindrome factors:
\[
 A p_1\rev A,
 \quad (\rev A)^{-1}p_2A^{-1},
 \quad A p_3\rev A,
 \quad (\rev A)^{-1}p_4A^{-1},\ldots .
\]
Their adjacent boundary pieces cancel in the group.  With two factors the
product is
\[
 A p_1p_2A^{-1},
\]
whereas with three factors it is
\[
 A p_1p_2p_3\rev A.
\]
Continuing in this way, the full product is $AhA^{-1}$ when $r$ is even and
$Ah\rev A$ when $r$ is odd, which is exactly $Ah\rho_r(A)$.
\end{proof}

\begin{proposition}\label{prop:template-sufficient}
Every literal instance of a $k$-leaf Catalan template represents an element
of palindromic length at most $k$.
\end{proposition}

\begin{proof}
We induct on the tree.  A one-leaf template is just its leaf variable, hence
is a palindrome.  Now let an internal vertex $v$ have left and right
subtrees with $i$ and $j$ leaves.  By induction, their literal values are
products of at most $i$ and at most $j$ palindromes.  Concatenating these
factorizations gives at most $i+j=\lambda(v)$ palindromes.  If fewer occur,
insert identity palindromes so that there are exactly $\lambda(v)$ factors.
Lemma~\ref{lem:closure}, with the internal label $A_v$, then shows that
\[
 A_v\,\mathcal C(v_L)\,\mathcal C(v_R)\,
 \rho_{\lambda(v)}(A_v)
\]
is still a product of $\lambda(v)$ palindromes.  At the root this gives at
most $k$ palindromes.
\end{proof}

The first few templates make the recursion explicit.
For one leaf the template is simply $P$.  For two leaves there is one binary
tree and the template is
\[
       APQA^{-1}.
\]
For three leaves there are two tree shapes, giving
\[
       APBQRB^{-1}\rev A,
       \qquad
       ABPQB^{-1}R\rev A.
\]
Here $P,Q,R$ are palindromes and $A,B$ are arbitrary words.  These are the
first visible instances of the Catalan pattern: each internal word returns as
an inverse when it encloses an even number of leaves and as a reversal when
it encloses an odd number.

The converse is the main structural statement.

\begin{theorem}[Catalan template characterization]\label{thm:catalan}
Let $g\in F$, let $w$ be its reduced word, and let $k\ge1$.  Then the
following are equivalent.
\begin{enumerate}[label=(\roman*)]
\item $\pl(g)\le k$.
\item The word $w$ is a literal instance of $\mathcal C(\mathcal T)$ for
some ordered full binary tree $\mathcal T$ with $k$ leaves.
\end{enumerate}
Hence $C_{k-1}$ tree shapes suffice, where
\[
        C_{k-1}=\frac1k\binom{2k-2}{k-1}.
\]
No assertion is made that these templates remain pairwise irredundant after
empty substitutions are allowed.
\end{theorem}

\begin{proof}
The implication (ii)$\Rightarrow$(i) is
Proposition~\ref{prop:template-sufficient}.

Assume (i), and set $\eps=k\bmod2$ and $h=gt^\eps$.  By
Proposition~\ref{prop:bounded},
\[
        m:=\ell_T(h)\le k.
\]
Choose a reduced Coxeter word $v$ for $h$.  Since every Coxeter generator is
in the odd coset, $|v|\equiv\eps\equiv k\pmod2$.  By
Proposition~\ref{prop:delta-reflection}, there is an optimal noncrossing
matching on $v$ with exactly $m$ unmatched positions.  Lemma~
\ref{lem:matching-parity} therefore gives
\[
        m\equiv |v|\equiv k\pmod2.
\]
This is the parity statement needed for the padding step below; in
particular, $k-m$ is even.

Suppose first that $m>0$.  Proposition~\ref{prop:contour} gives an ordered
full binary contour tree with $m$ leaves whose root contour is the reduced
Coxeter word $v$.  Proposition~\ref{prop:project-contour}, applied from
initial state $0$, turns this contour into a literal $m$-leaf Catalan
template instance.  Corollary~\ref{cor:recovery} identifies the total output
with the reduced free word $w$.  Thus $w$ is an $m$-leaf Catalan-template
instance.

If $m=k$, we are done.  If $m<k$, then $k-m$ is a positive even number.  We
use an explicit padding operation that adds two leaves without changing the
literal word.  Given an $m$-leaf tree $\mathcal T$, let $\mathcal E$ be the
two-leaf tree (a cherry), and put
\[
        \mathcal T'=(\mathcal T,\mathcal E).
\]
Assign the empty word to the new root, to the internal vertex of
$\mathcal E$, and to both new leaf palindromes.  Then
$\mathcal C(\mathcal T')$ has exactly the same literal value as
$\mathcal C(\mathcal T)$, while the number of leaves has increased from
$m$ to $m+2$.  Iterating this operation $(k-m)/2$ times gives a $k$-leaf
template with the same literal value $w$.

It remains only the case $m=0$.  Reflection length zero means $h=1$.  Since
$m\equiv k\pmod2$, the integer $k$ is even, so $h=g$ and therefore $g=1$.
The reduced word $w$ is empty.  Assigning the empty word to every variable
of any $k$-leaf template gives exactly this empty word.  This completes the
proof.
\end{proof}

\paragraph{Consequences.}
For fixed $k$, there are only finitely many tree shapes, namely the Catalan
number $C_{k-1}$.  The arbitrary words on internal vertices absorb all
possible cancellation patterns between palindrome factors; the leaves carry
the actual palindrome data.  Thus the theorem is a normal form for the
\emph{reduced word} itself, not merely an existence statement about a product
in the group.

\begin{remark}\label{rem:arbitrary-state}
The quantification over both entry states in
Proposition~\ref{prop:project-contour} is essential.  The right child of an
internal vertex need not be entered in state $0$; formulating the projection
theorem for the actual propagated state removes this otherwise implicit
point from the proof of Theorem~\ref{thm:catalan}.
\end{remark}

\section{Four palindromes}

\paragraph{The five four-leaf shapes.}
A full binary tree with four ordered leaves has only three possible root
splittings:
\[
        1+3,\qquad 2+2,\qquad 3+1.
\]
A three-leaf subtree has two possible shapes, $1+2$ and $2+1$.  Hence the
three root splittings contribute $2+1+2=5$ trees in total.  This is the
concrete origin of the five forms; no additional case analysis is hidden in
the argument.

For $k=4$ there are $C_3=5$ ordered full binary trees.  Frid proposed the five forms below in \cite[Section~5]{Frid}.  Let the leaves be
$P,Q,R,S$ from left to right, and name the three internal variables
$A,B,C$ in preorder.  Formula \eqref{eq:catalan-template} gives the
following five forms:
\begin{align*}
&APBQCRSC^{-1}\rev B A^{-1},\\
&APBCQRC^{-1}S\rev B A^{-1},\\
&ABPQB^{-1}CRSC^{-1}A^{-1},\\
&ABPCQRC^{-1}\rev BSA^{-1},\\
&ABCPQC^{-1}R\rev BSA^{-1}.
\end{align*}
Here $P,Q,R,S$ are palindromes and $A,B,C$ are arbitrary words, with empty
words permitted.

The rule for the returning copy of an internal label is determined only by
the number of leaves below it.  A label surrounding two or four leaves
returns as its inverse; a label surrounding one or three leaves returns as
its reversal.  For example, in the middle shape $(PQ)(RS)$ the label $B$
surrounds the two leaves $P,Q$, which is why it returns as $B^{-1}$, while
the root label $A$ surrounds all four leaves and therefore returns as
$A^{-1}$.  This parity rule explains every inverse/reversal in the five
displayed expressions.

\paragraph{Expansion of the middle shape.}
Take the middle tree $(PQ)(RS)$.  The left internal word $B$ surrounds two
leaves, so it returns as $B^{-1}$; the right internal word $C$ also surrounds
two leaves, so it returns as $C^{-1}$; and the root word $A$ surrounds all
four leaves, so it returns as $A^{-1}$.  Reading the tree from left to right
therefore gives
\[
       A\,B\,P\,Q\,B^{-1}\,C\,R\,S\,C^{-1}\,A^{-1},
\]
which is exactly the third displayed form.  The other four formulas are read
from their trees in the same way.  Thus the formulas are not independent
case guesses: each is the literal expansion of one ordered tree shape.

\begin{corollary}[Four-palindrome classification]\label{cor:frid}
Let $w$ be a reduced word representing $g\in F$.  Then $\pl(g)\le4$ if
and only if $w$ is a literal instance of one of the five forms above.
\end{corollary}

\begin{proof}
For the forward implication, assume $\pl(g)\le4$.  Theorem~\ref{thm:catalan} gives a literal four-leaf Catalan
template.  The five ordered full binary trees on four leaves are exactly
\[
P(Q(RS)),\qquad P((QR)S),\qquad (PQ)(RS),\qquad
(P(QR))S,\qquad ((PQ)R)S.
\]
Applying the recursive rule \eqref{eq:catalan-template} to these five shapes
produces the five displayed expressions.

Conversely, each displayed expression is by construction a literal instance
of the corresponding four-leaf Catalan template.  Proposition~
\ref{prop:template-sufficient} therefore shows that the represented element
has palindromic length at most four.  This proves both directions.
\end{proof}

\section{Examples and further remarks}

\paragraph{Example.}
The next word looks unfavorable if one only searches for palindromic
subwords of its reduced representative.  Nevertheless it has palindromic
length four because cancellation between different palindrome factors is
allowed in the group.  The Coxeter matching model detects this efficiently
and also gives the lower bound, illustrating why the method is stronger than
literal cutting of the free word.

The algorithm is a group-theoretic algorithm rather than a monoid
factorization algorithm.  For instance, put
\[
        w=aba^{-1}b^{-1}ab\in F(a,b).
\]
As a word over the alphabet $\{a^{\pm1},b^{\pm1}\}$, $w$ has no
palindromic factor of length greater than one, so literal palindromic cutting
requires six factors.  In the group, however,
\[
        w=(a)(ba^{-1}b)(b^{-3})(bab),
\]
so $\pl(w)\le4$.  The lower bound is equally transparent in the Coxeter
model.  Writing $s_a=at$, $s_b=bt$ and $s_0=t$, free reduction gives
\[
        wt=(s_as_0s_b)^3.
\]
In this nine-letter reduced Coxeter word, a noncrossing equal-label matching
has at most two pairs.  Indeed, a pair of span six leaves room for at most one
pair inside and none outside, while if all pairs have span three then their
supporting intervals must be disjoint, so again at most two occur.  Two pairs
are attained, for example at positions $(1,4)$ and $(5,8)$.  Hence
$\ell_T(wt)=9-4=5$.  Proposition~\ref{prop:bounded} with $k=3$ gives
$\pl(w)>3$, and therefore
\[
        \pl(aba^{-1}b^{-1}ab)=4.
\]

The Catalan theorem should be viewed as a reduced-word characterization of
bounded palindromic length.  Frid's descriptions for lengths two and three are the
first two nontrivial instances; her five proposed forms for length four are
the next one.  The universal-Coxeter model avoids a case analysis of the
relative amounts of cancellation between neighboring palindromes: those
cancellations are encoded by the laminar structure of a noncrossing
matching.

There are several natural directions.  First, the cubic dynamic program may
admit faster implementations.  Second, it would be interesting to know how
much of the construction survives for other graph products or free
products, where palindromes are defined in syllable normal form.  Finally,
the reflection-length viewpoint suggests comparing palindromic metrics with
absolute metrics in other groups possessing canonical involutive
extensions.

\paragraph{Proof summary.}
For reference, the proof can be compressed to the following implications:
\[
\begin{array}{c}
\pl(g)\le k\\
\Downarrow\ \text{Proposition~\ref{prop:bounded}}\\
\ell_T(gt^{k\bmod2})\le k\\
\Downarrow\ \text{Proposition~\ref{prop:delta-reflection}}\\
\text{optimal matching with }m\le k\text{ unmatched positions}\\
\Downarrow\ \text{if }m=0\text{, handle the empty word separately; if }m>0\\
\Downarrow\ \text{Proposition~\ref{prop:contour}}\\
\text{an }m\text{-leaf binary contour}\\
\Downarrow\ \text{Proposition~\ref{prop:project-contour}}\\
\text{a literal }m\text{-leaf Catalan template}\\
\Downarrow\ \text{Lemma~\ref{lem:matching-parity} and }(+2)\text{ padding}\\
\text{a literal }k\text{-leaf Catalan template}.
\end{array}
\]
The reverse implication is Proposition~\ref{prop:template-sufficient}.  For
$k=4$, the five possible tree shapes give Corollary~\ref{cor:frid}.

\section{Formal verification and scope}

The main mathematical result is Theorem~\ref{thm:catalan};
Corollary~\ref{cor:frid} is its four-leaf specialization.  Its logical chain
is entirely contained in the preceding sections: palindromes are converted to reflections, Dyer's theorem converts
reflection length to unmatched positions, the matching is converted to a
literal contour, the two-state projection recovers the original reduced free
word, and the Catalan closure lemma gives the converse implication.

The companion Lean 4 development included as ancillary material formalizes
the four-palindrome classification end to end.  Its final theorem is stated
for every finite rank $n$ and every reduced word $w$ over the basis of
\texttt{FreeGroup (Fin n)}.  On the left-hand side, palindromic length is
defined using multiplication in the free group, with each factor tested by
whether its unique reduced representative is a palindrome.  On the
right-hand side, the reduced word itself is required to be literally equal to
one of the five displayed forms, with all substitution words lying in the
same finite basis.

Thus the formal statement matches Corollary~\ref{cor:frid} at the
reduced-word level; no noncrossing matching, Coxeter representative,
projection, or Catalan representation is assumed as additional input.  The
formalization also contains the palindrome/reflection bridge, the
reflection/matching equivalence, the matching-to-Catalan construction, the
projection interfaces, and the literal Catalan semantics used in the proof.
A general-$k$ Catalan interface is also present in the development, but the
formal verification claim made here is only for the four-palindrome
classification stated above.

The audited source compiles without errors, warnings, unresolved goals, or
proof placeholders, and a Lean kernel replay succeeds.  The final theorem
depends only on the standard Lean/mathlib axioms \texttt{propext},
\texttt{Classical.choice}, and \texttt{Quot.sound}; no additional
mathematical axioms are introduced.  We do not claim formal verification of
the asymptotic complexity analysis of the algorithm.

\subsection*{Summary of the representations}
For quick reference, the same object is viewed in four successive ways:
\[
\begin{array}{ccl}
\text{free group} &:& g\text{ with reduced word }w,\\[2mm]
\text{Coxeter extension} &:& gt^{\eps}\text{ with reduced Coxeter word }v,\\[2mm]
\text{matching model} &:& v\text{ with }m\text{ unmatched positions},\\[2mm]
\text{Catalan model} &:& w\text{ as an }m\text{-leaf literal template}.
\end{array}
\]
When $m>0$, the number $m$ is simultaneously the reflection length of the
Coxeter lift and the number of leaves before padding.  The case $m=0$ is the
identity case and is handled by the all-empty template.  Padding adds two
empty leaves at a time and therefore changes neither the literal free word
nor its group element.  This is why a bound by $k$ produces a template with
exactly $k$ leaves once parity is taken into account.


\begin{thebibliography}{99}

\bibitem{BST}
V.~G. Bardakov, V. Shpilrain and V. Tolstykh,
\newblock On the palindromic and primitive widths of a free group,
\newblock \emph{J. Algebra} \textbf{285} (2005), 574--585.
\newblock arXiv:math/0311257; DOI: 10.1016/j.jalgebra.2004.11.003.

\bibitem{Dyer}
M.~J. Dyer,
\newblock On minimal lengths of expressions of Coxeter group elements as
products of reflections,
\newblock \emph{Proc. Amer. Math. Soc.} \textbf{129} (2001), 2591--2595.
\newblock DOI: 10.1090/S0002-9939-01-05876-2.

\bibitem{Frid}
A.~E. Frid,
\newblock Small palindromic lengths in free groups and word equations with
antimorphisms,
\newblock preprint, arXiv:2512.10024 (2025).
\newblock DOI: 10.48550/arXiv.2512.10024.

\bibitem{KimEtAl}
S.-h. Kim, T. Koberda, J. Lee, K. Ohshika, S. P. Tan and X. Gao,
\newblock Shapes of hyperbolic triangles and once-punctured torus groups,
\newblock \emph{Math. Z.} \textbf{299} (2021), 2103--2130.
\newblock DOI: 10.1007/s00209-021-02745-3.

\bibitem{Kourovka}
E.~I. Khukhro and V.~D. Mazurov (eds.),
\newblock \emph{Unsolved Problems in Group Theory. The Kourovka Notebook},
21st ed.,
\newblock Sobolev Institute of Mathematics, Novosibirsk, 2026.
\newblock arXiv:1401.0300.

\bibitem{PiggottRuane}
A. Piggott and K. Ruane,
\newblock Normal forms for automorphisms of universal Coxeter groups and
palindromic automorphisms of free groups,
\newblock \emph{Internat. J. Algebra Comput.} \textbf{20} (2010), 1063--1086.
\newblock DOI: 10.1142/S0218196710006035.

\bibitem{Saarela}
A. Saarela,
\newblock Palindromic length in free monoids and free groups,
\newblock in \emph{Combinatorics on Words}, Lecture Notes in Comput. Sci.
\textbf{10432}, Springer, 2017, 203--213.
\newblock DOI: 10.1007/978-3-319-66396-8\_19.

\end{thebibliography}
\end{document}